\documentclass[a4paper, twoside,12pt]{article}
 \usepackage{fancyhdr}
 \usepackage[english]{babel}        %% adapte le style article aux conventions francophones
\usepackage[T1]{fontenc}          %% permet d'utiliser les caractï¿½res accentuï¿½s
\usepackage{graphicx}             %% permet d'importer des graphiques au format .EPS (postscript)
\usepackage{framed}
\usepackage[margin=1in]{geometry}
\usepackage{graphicx}
\usepackage{amsmath}
\usepackage{amsfonts}   % math fonts
\usepackage{amssymb}    % extra math symbols
\usepackage{amsthm}
 \usepackage[english]{babel}        %% adapte le style article aux conventions francophones
\usepackage[T1]{fontenc}
\usepackage{dsfont}       %% permet d'utiliser les caractï¿½res accentuï¿½s
\usepackage{graphicx}             %% permet d'importer des graphiques au format .EPS (postscript)
\usepackage[margin=1in]{geometry}
\usepackage{graphicx}
\usepackage{amsfonts}   % math fonts
\usepackage{amsthm}
\usepackage{amsmath,amssymb,color,inputenc,euscript,graphicx,psfrag}

\newtheorem{thm}{Theorem}[section]

\newtheorem{lem}[thm]{Lemma}

\numberwithin{equation}{section}

\renewcommand{\thefootnote}

\author {B\'echir Amri }
\title{  About integral product formula for Jack polynomials of two variables. }
\date{}
\begin{document}

\maketitle
\begin{abstract}
 We present an integral product formula for Jack polynomials of two variables, extending that of zonal polynomials. It provides another way to find  the  explicit integral representation for the generalized  Bessel function of type $ B_2 $, as presented in \cite{D}.
     \end{abstract}
\section{Introduction}
 Jack polynomials $J_\lambda^\alpha$ are family of homogenous  symmetric polynomials indexed by   partitions $\lambda$ and depend on a parameter $\alpha$.
     They were first  introduced by the statistician H. Jack in  \cite{J} as generalisation   of zonal spherical functions on some symmetric spaces.
Zonal symmetric polynomials are the case $\alpha=2$ which can be  constructed in term of eigenvalues of a symmetric matrix \cite{JA} and interpreted as the zonal spherical functions of the Gelfand pair   $(GL(n), O(n))$ \cite{MA}.   Motivated by special values of  the parameters $ \alpha =1, 1/2, 2$   many properties of the zonal spherical functions  can be extended to Jack polynomials \cite{ST}.
 This      work  present an extension of the product formula, one  of the interesting  property   for   zonal polynomials,
to  Jack polynomials of second order.  This is carried out by  considering  the general   partial differential equation for Jack polynomials  where   two eigenvalues are considered. Basic references for the theory of Jack polynomials are here \cite{ST, MA, OO}.
\par Let  $\lambda = ( \lambda_1, \lambda_2,...,\lambda_n) \in \mathbb{N}^n, \lambda_1 \geq  \lambda_2...\geq \lambda_n$  be a partition of
$$|\lambda|=  \sum_{i=1}^n\lambda_i.$$
   The monomial symmetric functions  $ m_\lambda(x)$ are defined by
   $$m_\lambda(x)=\sum_{\sigma\in S_n} x_{\sigma(1)}^{\lambda_1}x_{\sigma(2)}^{\lambda_2}...x_{\sigma(n)}^{\lambda_n}$$
   For two   partitions  $\lambda$ and  $\mu$
we denote $ \lambda> \mu$ if $|\lambda|= |\mu|$  and  for all $i \geq  1$,  $ \lambda_1+\lambda_2+...+\lambda_i \geq  \mu_1+\mu_2+...+\mu_i$. For a given partition $\lambda$, the Jack polynomials $P_\lambda^k$ attached with parameter $k>0$ can be defined as the unique   symmetric eigenfunction of the operator
$$\sum_{i=1}^nx_i^2\frac{\partial^2}{\partial^2x_i}+2k\sum_{i\neq j}\frac{x_i^2}{x_i-x_j}\frac{\partial}{\partial x_i}$$
which has an expansion of the form
$$ m_\lambda+\sum_{\mu<\lambda}a_{\mu\lambda}m_{\mu},$$
with the eigenvalue
\begin{equation}\label{e}
  e_\lambda= \sum_{i=1}^n\lambda_i(\lambda_i+2k (n-i)-1).
\end{equation}
 Here the  parameter  $k$  is related to Jack's  $\alpha$ by $k=1/ \alpha$.
There is an  integral representation  of Jack polynomials \cite{OO}  giving  a
way to obtain the polynomial in n variables from the polynomial of n-1 variables. It states that
for $\lambda$ a partition of at most $n-1$ parts
\begin{eqnarray}\label{rec}
 P_\lambda^k(x)=&& \prod_{i=1}^{n-1}\frac{\Gamma(\lambda_i+(n-i)k)}{\Gamma(\lambda_i+(n-i+1)k)\Gamma(k)}\prod_{1\leq i<j\leq n}(x_i-x_j)^{1-2k}
\\&&\times\int_{\nu\prec x} P_\lambda^k(\nu)\prod_{1\leq i<j\leq n-1}(\nu_i-\nu_j) \prod_{i=1}^{n-1} \prod_{j=1}^n|x_i-\nu_j|^{k-1}
\end{eqnarray}
where $\nu\prec x$ means that $x_1\leq \eta_1\leq x_2\leq....\leq \eta_{n-1}\leq x_n$.
 There is another  way of defining Jack polynomials  constructed by Gram-Schmidt  orthogonalization  relative to scalar   product on the ring of symmetric polynomials, see\cite{ST}.
  \par When the parameter $k=1/2$, Jack polynomials $P_\lambda^k$ can be connected to    zonal spherical  polynomials $Z_\lambda$  on the space $\Sigma_n^+$  of real symmetric positive definite matrix via
   $$ \frac{Z_\lambda(x)}{Z_\lambda(I)}=\frac{P_\lambda^k(x_1,x_2...,x_n)}{P_\lambda^k(1,1...,1)}$$
  for $x\in \Sigma_n^+$  with eigenvalues $x_1,x_2...,x_n$.
   From the theory of zonal spherical function
   $Z_\lambda$ satisfies the following product formula
   \begin{equation}\label{zonal}
     \frac{Z_\lambda(s) Z_\lambda(t)}{Z_\lambda(1)}=\int_{O(n)}Z_\lambda(sktk')dk=2\int_{SO(n)}Z_\lambda(s^{1/2}ktk's^{1/2})dk
 \end{equation}
where $dk$ is the Haar measure of  the orthogonal group $O(n)$.
\section{The case $n=2$}
 Our aim is  to find a generalization  of the product formula
  (\ref{zonal})  in terms of Jack polynomials $P_\lambda^k$. Let
$$s= \left(
       \begin{array}{cc}
         x_1 & 0\\
         0 & x_2 \\
       \end{array}
     \right),\qquad t= \left(
       \begin{array}{cc}
         y_1 & 0\\
         0 & y_2 \\
       \end{array}
     \right)
$$
 where   $x_1, x_2,y_2,y_2$ are nonnegative real numbers. Matrices of
  $SO(2)$  are given by

$$k=k_\theta= \left(
      \begin{array}{cc}
        \cos \theta & -\sin \theta \\
       \sin \theta & \cos \theta \\
      \end{array}
    \right),  \qquad \theta\in [0,2\pi]
$$
 and the Haar measure  $dk = (2\pi)^{-1} d\theta$.   It is straightforward to verify that
$$s^{1/2}ktk's^{1/2}=\left(
                       \begin{array}{cc}
                         x_1y_1\cos^2\theta+x_1y_2\sin^2\theta &  \sqrt{x_1x_2}(y_1-y_2)\cos \theta\sin \theta\\
                          \sqrt{x_1x_2}(y_1-y_2)\cos \theta\sin \theta  &   x_2y_2\cos^2\theta+x_2y_1\sin^2\theta \\
                       \end{array}
                     \right)
$$
and  the corresponding  eigenvalues are given by
\begin{eqnarray*}
  X_1(\theta)=   && \frac{1}{2}\Bigg\{(x_1y_1+x_2y_2)\cos^2\theta+(x_1y_2+x_2y_1)\sin^2\theta  +\\  &&
\sqrt{\Big((x_1y_1+x_2y_2)\cos^2\theta+(x_1y_2+x_2y_1)\sin^2\theta\Big)^2-4 x_1x_2 y_1y_2}\;\Bigg\}
\end{eqnarray*}
and
\begin{eqnarray*}
   X_2(\theta)= && \frac{1}{2}\Bigg\{(x_1y_1+x_2y_2)\cos^2\theta+(x_1y_2+x_2y_1)\sin^2\theta  -\\  &&
\sqrt{\Big((x_1y_1+x_2y_2)\cos^2\theta+(x_1y_2+x_2y_1)\sin^2\theta\Big)^2-4 x_1x_2 y_1y_2}\;\Bigg\}.
\end{eqnarray*}
 It will be convenient later  to set   $ \cos^2\theta = (u+1)/2$ and we put
\begin{eqnarray}\label{v1}
\alpha= \frac{1}{2}\Big( x_1+x_2)(y_1+y_2)  + (x_1-x_2)(y_1-y_2)u\Big) \nonumber
 \\ a=\frac{1}{2}\Big\{ y_1+y_2+ (y_1-y_2)u\Big),\qquad \overline{a}=\frac{1}{2}\Big\{ y_1+y_2 -(y_1-y_2)u\Big).
 \end{eqnarray}
    We have then
    \begin{eqnarray}\label{v2}
     X_1(u)=   \frac{1}{2}\Big( \alpha +
\sqrt{\alpha^2-4 x_1x_2 y_1y_2}\;\Big )\; \text{and}  \;  X_2 (u)=   \frac{1}{2}\Big( \alpha -
\sqrt{\alpha^2-4 x_1x_2 y_1y_2}\; \Big)
\end{eqnarray}
  The following basic result gives an  integral product formula for $P_\lambda^k$.
 \begin{thm}\label{th1} For a partition $\lambda$ and for  $x=(x_1,x_2) $,  $y=(y_1,y_2)$,
 \begin{equation}\label{IF}
   \frac{P_\lambda^k(x) P_\lambda^k(y)}{P_\lambda^k(1)}=  \frac{ \Gamma(k+1/2)}{\Gamma(k)\sqrt{\pi}}\int_{-1}^1P_\lambda^k( X_1(u), X_2(u)) (1-u^2)^{k-1}du.
 \end{equation}
  \end{thm}
\begin{proof}
Fixing the variable  $y$ and  consider the function
  $$H_k(x)=  \frac{ \Gamma(k+1/2)}{\Gamma(k)\sqrt{\pi}} \int_{-1}^1P_\lambda^k( X_1(u), X_2(u)) (1-u^2)^{k-1}du.$$
Our proof consists in  showing in a first part that $H_k$ is eigenfunction of   the Laplace-Beltrami-type operator,
 \begin{equation}\label{LB}
\Lambda_k = x_1^2\frac{\partial^2}{\partial x_1^2}+x_2^2\frac{\partial^2}{\partial x_2^2}+2k\Big\{ \frac{x_1^2}{x_1-x_2}  \frac{\partial}{\partial x_1}+
 \frac{x_2^2}{x_2-x_1}  \frac{\partial}{\partial x_2}\Big\}.
\end{equation}
More precisely we show that
\begin{equation}\label{eig}
  \Lambda_k (H_k)= e_\lambda  H_k.
\end{equation}
where $e_\lambda$ is given by (\ref{e}).
A direct computation yields
 \begin{eqnarray}
&&\nonumber x_1^2\frac{\partial^2}{\partial x_1^2}P_\lambda^k( X_1, X_2)+ x_2^2\frac{\partial^2}{\partial x_2^2}P_\lambda^k( X_1, X_2)
 \\&&\nonumber\qquad \qquad \qquad \qquad=\left( x_1^2\frac{\partial^2 X_1 }{\partial x_1^2}+x_2^2\frac{\partial^2 X_1 }{\partial x_2^2} \right)
 \left(\frac{\partial}{\partial X_1}P_\lambda^k( X_1, X_2)-\frac{\partial}{\partial X_2}P_\lambda^k( X_1, X_2)\right)
 \\ && \nonumber\qquad \qquad \qquad \qquad+\left(\left( x_1\frac{\partial X_1 }{\partial x_1}\right)^2+\left(x_2\frac{\partial X_1 }{\partial x_2} \right)^2 \right) \frac{\partial^2}{\partial X_1^2}P_\lambda^k( X_1, X_2)\\ &&\nonumber\qquad \qquad \qquad \qquad+\left(\left( x_1\frac{\partial X_2 }{\partial x_1}\right)^2+\left(x_2\frac{\partial X_2 }{\partial x_2} \right)^2 \right) \frac{\partial^2}{\partial X_2^2}P_\lambda^k( X_1, X_2)
 \\ && \qquad \qquad \qquad \qquad+\left( 2x_1^2\frac{\partial X_1 }{\partial x_1}\frac{\partial X_2 }{\partial x_1}+ 2x_2^2\frac{\partial X_1 }{\partial x_2}\frac{\partial X_2 }{\partial x_2}\right)\frac{\partial^2}{\partial X_1\partial X_2}P_\lambda^k( X_1, X_2). \label{e_1}
\end{eqnarray}
 With the notations of (\ref{v1}) and   (\ref{v2}) we have
 \begin{equation}\label{X1}
   \frac{\partial X_1 }{\partial x_1}=  \frac{aX_1-x_2y_1y_2}{X_1-X_2},
\quad  \frac{\partial X_1 }{\partial x_2}=  \frac{\overline{a}X_1-x_1y_1y_2}{X_1-X_2}
 \end{equation}
 and
 \begin{equation}\label{X2}
     \frac{\partial X_2 }{\partial x_1}=  \frac{aX_2-x_2y_1y_2}{X_2-X_1}, \quad \frac{\partial X_2 }{\partial x_2}=  \frac{\overline{a}X_2-x_2y_1y_2}{X_2-X_1}.
      \end{equation}
It follows that
$$\frac{x_1^2}{x_1-x_2}\frac{\partial X_1 }{\partial x_1}+\frac{x_2^2}{x_2-x_1}\frac{\partial X_1 }{\partial x_2}
 = \frac{X_1^2}{X_1-X_2}+\frac{ux_1x_2(y_1-y_2) X_1}{(x_1-x_2)(X_1-X_2) }$$
and
$$\frac{x_1^2}{x_1-x_2}\frac{\partial X_2 }{\partial x_1}+\frac{x_2^2}{x_2-x_1}\frac{\partial X_2 }{\partial x_2}
 = \frac{X_2^2}{X_2-X_1}-\frac{ux_1x_2(y_1-y_2) X_2 }{(x_1-x_2) (X_1-X_2)}$$
Apply  the second  part of the operator (\ref{LB}) to $P_\lambda^k(X_1,X_2)$ yields
 \begin{eqnarray}\label{e1}
 &&\frac{x_1^2}{x_1-x_2} \frac{\partial}{\partial x_1}P_\lambda^k (X_1, X_2)+ \frac{x_2^2}{x_2-x_1} \frac{\partial}{\partial x_2}P_\lambda^k (X_1, X_2)\\&&=\frac{X_1^2}{X_1-X_2} \frac{\partial}{\partial X_1}P_\lambda^k (X_1, X_2)+ \frac{X_2^2}{X_2-X_1} \frac{\partial}{\partial X_2}P_\lambda^k (X_1, X_2)
+ \frac{2x_1x_2  }{(x_1-x_2)^2}u\frac{\partial}{\partial u}P_\lambda^k (X_1, X_2)\nonumber,
 \end{eqnarray}
since one   can   easily see
$$\frac{\partial}{\partial u}P_\lambda^k (X_1, X_2)=     \frac{(x_1-x_2) (y_1-y_2)  }{  2(X_1-X_2)}\left(  X_1\frac{\partial}{\partial X_1}P_\lambda^k (X_1, X_2)
- X_2\frac{\partial}{\partial X_1}P_\lambda^k (X_1, X_2) \right)$$
In the same way using (\ref{X1}) and (\ref{X2}) we obtain
 \begin{eqnarray}
        \left( x_1\frac{\partial X_1 }{\partial x_1}\right)^2+\left(x_2\frac{\partial X_1 }{\partial x_2} \right)^2  & = &  X_1^2 -\frac{x_1x_2 (y_1-y_2)^2X_1^2}{2(X_1-X_2)^2}(1-u^2)\label{e3}
            \\   \left( x_1\frac{\partial X_2 }{\partial x_1}\right)^2+\left(x_2\frac{\partial X_2 }{\partial x_2} \right)^2 &=&
    X_2^2 -\frac{x_1x_2 (y_1-y_2)^2X_2^2}{2(X_1-X_2)^2}(1-u^2) \label{e4}
    \\   2x_1^2\frac{\partial X_1 }{\partial x_1}\frac{\partial X_2 }{\partial x_1}+ 2x_2^2\frac{\partial X_1 }{\partial x_2}\frac{\partial X_2 }{\partial x_2} &=& \frac{x_1^2x_2^2y_1y_2(y_1-y_2)^2}{(X_1-X_2)^2}(1-u^2)\label{e5}
    \\ x_1^2\frac{\partial^2 X_1 }{\partial x_1^2}+x_2^2\frac{\partial^2 X_1 }{\partial x_2^2}& =&-\frac{ x_1x_2y_1y_2(y_1-y_2)^2}{ (X_1-X_2)^3}(1-u^2)\label{e6}
\end{eqnarray}
 The above computation of   derivatives
involve again the second derivative of $P_\lambda^k(X_1, X_2)$  with respect to $u$,
 \begin{eqnarray*}
   \frac{\partial^2}{\partial u^2}P_\lambda^k(X_1,X_2)= \frac{\partial^2 X_1}{\partial u^2} \left(\frac{\partial}{\partial X_1}P_\lambda^k(X_1, X_2)-
\frac{\partial}{\partial X_2}P_\lambda^k (X_1,X_2)  \right)\qquad\qquad\qquad\qquad\qquad\\ +
 \left(\frac{\partial X_1}{\partial u}\right)^2\frac{\partial^2}{\partial X_1^2}P_\lambda^k(X_1, X_2)+
\left(\frac{\partial X_2}{\partial u}\right)^2\frac{\partial^2}{\partial X_2^2}P_\lambda^k(X_1, X_2)+2\frac{\partial X_1}{\partial u}\frac{\partial X_2}{\partial u}\frac{\partial^2}{\partial X_1\partial X_2 }P_\lambda^k(X_1, X_2)
\end{eqnarray*}
where from  (\ref{v2}) it follows that
 \begin{eqnarray}
  \frac{\partial^2 X_1}{\partial u^2}&=& -\frac{x_1x_2y_1y_2(x_1-x_2)^2(y_1-y_2)^2}{(X_1-X_2)^3}\label{e7}
  \\ \left( \frac{\partial X_1}{\partial u}\right)^2&=& \frac{ (x_1-x_2)^2(y_1-y_2)^2 X_1^2}{4(X_1-X_2)^2}\label{e8}
  \\ \left( \frac{\partial X_2}{\partial u}\right)^2&=& \frac{ (x_1-x_2)^2(y_1-y_2)^2 X_2^2}{4(X_1-X_2)^2}\label{e9}
  \\ \frac{\partial X_1}{\partial u}\frac{\partial X_2}{\partial u}
  &=& -\frac{x_1x_2y_1y_2 (x_1-x_2)^2(y_1-y_2)^2  }{4(X_1-X_2)^2}.\label{e10}
 \end{eqnarray}
 Now  by  integration by part
\begin{eqnarray*}
  &&2k\int_{-1}^1 \left\{\frac{x_1^2}{x_1-x_2} \frac{\partial}{\partial x_1}P_\lambda^k (X_1, X_2)+ \frac{x_2^2}{x_2-x_1} \frac{\partial}{\partial x_2}P_\lambda^k (X_1, X_2)\right\}(1-u^2)\;du
\\&& \qquad\qquad =\int_{-1}^1 \left\{\frac{X_1^2}{X_1-X_2} \frac{\partial}{\partial X_1}P_\lambda^k (X_1, X_2)
+ \frac{X_2^2}{X_2-X_1} \frac{\partial}{\partial X_2}P_\lambda^k (X_1, X_2)\right\}(1-u^2)^{k-1}\;du
\\ &&\qquad\qquad+ \frac{2x_1x_2  }{(x_1-x_2)^2}\int_{-1}^1 \frac{\partial^2}{\partial u^2}P_\lambda^k (X_1, X_2)(1-u^2)^{k}\;du.
\end{eqnarray*}
Therefore combine (\ref{e_1}), (\ref{e1}) and  (\ref{e3}) ----- (\ref{e10})  to obtain  (\ref{eig}).
\par Now noting that
     $H_k$ is a homogeneous symmetric polynomial (making use the  change of variable $u$ with   $-u$).  However, for $\lambda=(\lambda_1,\lambda_1)$    we   write
    $$m_\lambda(X_1,X_2)=(X_1X_2)^{\lambda_2}( X_1^{\lambda_1-\lambda_2}+X_2^{\lambda_1-\lambda_2})=(x_1x_2y_1y_2)^{\lambda_2}( X_1^{\lambda_1-\lambda_2}+X_2^{\lambda_1-\lambda_2})$$
   and we set
  $$ G_k(x_1,x_2)=\int_{-1}^{1}( X_1^{\lambda_1-\lambda_2}(u)+X_2^{\lambda_1-\lambda_2}(u))(1-u^2)^{k-1}\;du.$$
 Clearly $G_k$ ia a    homogeneous symmetric polynomial of degree $\lambda_1-\lambda_2$,  and  it can be written as  finite linear combination of    $m_\mu(x_1,x_2)$  with  $\mu=(\mu_1,\mu_2)$,
 $\mu_1+\mu_2=\lambda_1-\lambda_2$. This implies that $H_k$ is  a finite linear combination of
 $$ (x_1x_2)^{\lambda_2}m_\mu(x_1,x_2)=m_{\mu_1+\lambda_2, \mu_2+\lambda_2}(x_1,x_2)$$
 As $|\lambda|=\mu_1+\lambda_2+\mu_2+\lambda_2$ and $\lambda_2\leq \lambda_2+\mu_2$  then we  have  $ \mu_1+\lambda_2\leq \lambda_1$,
which imply that   $(\mu_1+\lambda_2, \mu_2+\lambda_2)\leq (\lambda_1,\lambda_2)$ in dominance order.
Thus   we   conclude that $H_k$ takes the form 
$$H_k= \sum_{\mu\leq \lambda}h_{\mu,\lambda}m_\mu$$
 and it can be written as
$$H_k(x_1,x_2)= A(y_1,y_2)P_\lambda^k(x_1,x_2)$$
But since $H_k(1,1)=P_\lambda^k(y_1,y_2)$  then one has
$$ A(y_1,y_2)=\frac{P_\lambda^k(y_1,y_2)}{P_\lambda^k(1,1)}.$$
This completes  the proof of   Theorem \ref{th1}.
\end{proof}
\section{Applications}
In a first application we show that  the integral product formula  (\ref{IF}) can be used  to  find  a known     integral representation  for generalized Bessel function associated with a root system of type  $B_2$, given in  \cite{C, D}.
\par   Define the  hypergeometric function  of two arguments,
 $$_0F_1(\mu,x,y)=\sum_{\lambda}\frac{1}{|\lambda|![\mu]_\lambda^k}\: \frac{C_\lambda^k (x)C_\lambda^k (y)}{C_\lambda^k (1)}$$
where $C_\lambda^k$ are   Jack polynomials associated  with parameter $k$ and normalized such that
$$(x_1+x_2+...+x_n)^n= \sum_{|\lambda|=n}C_\lambda^k(x), \qquad x=(x_1,x_2,...,x_n)$$
and
$$[\mu]_\lambda^k= \prod_{j=1}^n\left(\mu- k(j-1)\right)_{\lambda_j}.$$
 The relationship between $P_\lambda^k$ and $C_\lambda$ is given by
$$C_\lambda^k= \frac{C_\lambda(1)}{P_\lambda^k(1)}\; P_\lambda^k= \frac{  |\lambda|!}{h_k(\lambda)}P_\lambda^k$$
with
$$ h_k(\lambda)=\prod_{1\leq i\leq \ell(\lambda);\; 1\leq j\leq \lambda_i} ( \lambda_i-j+1+k( \lambda_j'-i))$$
Then
\begin{equation}\label{0f1}
 _0F_1(\mu,x,y)=\sum_{\lambda}  \frac{ 1} {[\mu]_\lambda^\alpha  h_k(\lambda)} \frac{P_\lambda^k(x)P_\lambda^k(y)}{P_\lambda^k(1)}.
\end{equation}
For $x=(x_1,x_2,...,x_n)\in \mathbb{R}^n$ we put $x^2=(x_1^2,x_2^2,...,x_n^2)$.
 In  [\cite{R2}, Prop. 4.5],   R\"{o}sler   has shown  that   generalized Bessel function associated with a root system of type  $B_n$,
 $$R=\{\pm e_i,\; \pm e_i\pm e_j;\; 1\leq i,j\leq n \}$$
 and a multiplicity $ \kappa=(\kappa_1,\kappa_2)$  is connected to hypergeometric function $_0F_1$ by
 \begin{equation}\label{JF}
   J^\kappa_{B_n}( x,y)=_0F_1(\mu,x^2/,y^2/2),
 \end{equation}
with  $k=\kappa_2$ and
$$  \mu= \kappa_1+\kappa_2(n-1)+\frac{1}{2}.$$
Here  $\kappa_1$ and $\kappa_2$ are
the values of  $\kappa$ on the roots $\pm e_i$  and $\pm e_i\pm e_j$, respectively.
\begin{lem} \label{lem4}let $x=(x_1,x_2)$ and $e_1=(1,0)$. Then
 $$ _0F_1(\mu,x,e_1)= \frac{\Gamma(k+1/2)}{\Gamma(k)\sqrt{\pi}}\int_{-1}^1  \mathcal{I}_{\mu-1}\left(\sqrt{2(x_1+x_2+v(x_1-x_2))}\right) (1-v^2)^{k-1} \;dv,$$
where $ \mathcal{I}_{\mu-1} $
 is the normalized modified Bessel function of the first kind and of order  $\mu-1$
$$ \mathcal{I}_{\mu-1}(t) =  \Gamma(\mu ) \sum_{m=0}^\infty
\frac{(t/2)^{2m} }{m!\Gamma(m+\mu)},\qquad t\in \mathbb{R}.$$
\end{lem}
\begin{proof}
   In the two dimensional case,  the formula (\ref{rec})   can be written as
   \begin{equation}\label{n=2}
      P_\lambda^k(x_1,x_2) = \frac{\Gamma (m+2k)}{ 2^{2k-1}\Gamma( k)\Gamma(m+k)}\; (x_1x_2)^{\ell}
  \int_{-1}^{1} \left(\frac{x_1+x_2}{2}+u\frac{x_1-x_2}{2}\right)^{m}(1-u^2)^{k-1}  \;du.
   \end{equation}
 where  $ \ell=\lambda_2 $ and  $m=\lambda_1-\lambda_2$.
We  have
$$h_k(\lambda)= ( m+1+k)_{\ell}\;m!\;\ell!\;,\quad [\mu]_{\lambda}=(\mu+\ell)_{m}\;(\mu)_{\ell}\;( \mu-k)_{\ell} $$
and according to ( 6.4) of \cite{OO}
$$P_{ (m,0) }(1)=(m+k)_{k}\frac{\Gamma(k)}{\Gamma(2k)}=\frac{\Gamma (m+2k)\Gamma(k)}{\Gamma(m+k)\Gamma(2k)}.$$
 Hence one can write
\begin{eqnarray*}
   _0F_1(\mu, x,y)=\sum_{\ell=0}^\infty \frac{  (x_1 x_2 y_1 y_2)^{\ell} }{ (  \mu-k )_{ \ell }\; (\mu)_{\ell}\;\ell! }
   \sum_{ m=0}^\infty \frac{1}{   ( \mu +\ell)_{ m }\;
   ( k_2 +1+m )_{\ell}\; m!}  \frac{P_{(m,0)}(x)P_{ (m,0) }(y)}{P_{ (m,0) }(1)}
\end{eqnarray*}
 and  we obtain  that \\  \\
$\displaystyle{_0F_1(\mu,x,e_1)}$
\begin{eqnarray*}
  &=& \sum_{ m=0}^\infty \frac{1}{   ( \mu)_{ m }
     m!}  \frac{P_{(m,0)}(x)}{P_{ (m,0) }(1)}
     \\&= &
    \frac{\Gamma(\mu)}{{2^{2k-1}}\beta(k,k) } \int_{-1}^{1}\sum_{ m=0}^\infty \frac{1}{   \Gamma(\mu+m)
     m!}  \left(\frac{x_1+x_2}{2}+u\frac{x_1-x_2}{2}\right)^{m}(1-u^2)^{k-1}  \;du
     \\&=&  \frac{\Gamma(k+1/2)}{\Gamma(k)\sqrt{\pi}}\int_{-1}^{1} \mathcal{I}_{\mu-1}(\sqrt{2(x_1+x_2+u(x_1-x_2)})(1-u^2)^{k-1}  \;du.
\end{eqnarray*}
which is the desired result.
   \end{proof}
\begin{lem}\label{lem3}
  In the two dimensional case,  if we  let $r$ be the rotation
  $$r= \frac{1}{\sqrt{2}} \left(
                            \begin{array}{cc}
                              1 & 1 \\
                              -1& 1 \\
                            \end{array}
                          \right)$$

   then we have $$J^{ \kappa}_{B_2}( x,y)=J^{ \kappa'}_{B_2}(r.x,r.y),$$
   where $\kappa=(\kappa_1,\kappa_2) $ and  $\kappa'=(\kappa_2,\kappa_1)$.
   \end{lem}
  \begin{proof}
   In $\mathbb{R}^2$ the Dunkl operators of type $B_2$ are  given  by
    \begin{eqnarray*}
     T_1^\kappa f(x)=&& \frac{\partial f}{\partial x_1}(x) +\kappa_1\frac{f(x_1,x_2)-f(-x_1,x_2)}{x_1}
     +\kappa_2\frac{f(x_1,x_2)-f(x_2,x_1)}{x_1-x_2}
     \\&+&\kappa_2\frac{f(x_1,x_2)-f(-x_2,-x_1)}{x_1+x_2}
    \end{eqnarray*}
    and
      \begin{eqnarray*}
     T_2^\kappa f(x)=&& \frac{\partial f }{\partial x_2}(x) +\kappa_1\frac{f(x_1,x_2)-f(x_1,-x_2)}{x_2}
     -\kappa_2\frac{f(x_1,x_2)-f(x_2,x_1)}{x_1-x_2}
     \\&+&\kappa_2\frac{f(x_1,x_2)-f(-x_2,-x_1)}{x_1+x_2}.
    \end{eqnarray*}
    For $y=(y_1,y_2)$ the Dunkl kernel $E^\kappa_{B_2}(.,y)$ of type $B_2$ is the unique   solution  of the system
    \begin{eqnarray*}
    T_i^\kappa(f)(x)&=&y_i f(x),\quad i= 1,2
   \\ f(0)&=&1.
    \end{eqnarray*}
    The generalized Bessel function $J_B^{ \kappa}$ is  defined to be
    \begin{equation}\label{bess}
      J_{B_2}^{ \kappa}(x,y)=\sum_{w\in W_B }E^\kappa_{B_2}(w.x,y), \qquad x,y\in \mathbb{R}^2,
    \end{equation}
     where $W_B$ is the Weyl group for the root system of type $B_2$,
     $$R=\{\pm e_1,\pm e_2,\pm e_1\pm e_2\}.$$
  With the notation $r^{-1}.f(x)=f(r.x)$,   one can see that
    \begin{eqnarray*}
     T_1^\kappa(r^{-1}.f)(x)&=& \frac{1}{\sqrt{2}}T_1^{\kappa'}(f)(r.x)- \frac{1}{\sqrt{2}}T_2^{\kappa'}(f)(r.x)
      \\T_2^\kappa(r^{-1}.f)(x)&=& \frac{1}{\sqrt{2}}T_1^{\kappa'}(f)(r.x)+ \frac{1}{\sqrt{2}}T_2^{\kappa'}(f)(r.x)
    \end{eqnarray*}
      Now if we let  $f(x)=E^{\kappa'}_{B_2}(r.x,r.y)$ then $f$  satisfies
    \begin{eqnarray*}
     T_1^{\kappa}(f)(x)&=& \frac{y_1+y_2}{2}f(x) +\frac{y_1-y_2}{2}f(x)=y_1f(x),
      \\T_2^{\kappa}(f)(x)&=& \frac{-y_1+y_2}{2}f(x)-\frac{y_1-y_2}{2}f(x)=y_2f(x)\quad \text{and}\\  \;f(0)&=&1
    \end{eqnarray*}
 From which it follows that
 $$ f(x)=E^{\kappa'}_{B_2}(r.x,r.y)=E^{\kappa}_{B_2}(x,y)$$
 and  the lemma follows from (\ref{bess}).
 \end{proof}
 \begin{thm}
 The generalized Bessel function of type $B_2$ associated with a multiplicity $\kappa=(\kappa_1,\kappa_2) $ has the following integral representation
   \begin{equation}\label{}
     J_{B_2}^{ \kappa}(x,y)=c_\kappa\int_{-1}^1\int_{-1}^1\; \mathcal{I}_\mu \left( \sqrt{\frac{Z_{x,y}(u,v)}{2}}\right )(1-u^2)^{\kappa_2-1}(1-v^2)^{\kappa_1-1}\;du\;dv
   \end{equation}
   where
   $$ Z_{x,y}(u,v)= \sqrt{  (x_1^2+x_2^2)(y_1^2+y_2^2)+u (x_1^2-x_2^2)(y_1^2-y_2^2) +4v x_1x_2y_1y_2 }$$
   and
   $$c_\kappa=\frac{\Gamma(\kappa_1+1/2)\Gamma(\kappa_2+1/2)}{\pi\Gamma(\kappa_1)\Gamma( \kappa_2) }$$
 \end{thm}\label{th3}
 \begin{proof}
  Let us denote
  $$ A(u)= \Big( X_1(x^2/2,y^2/2 ,u),X_2(x^2/2,y^2/2, u)\Big)=(A_1(u),A_2(u)) \quad \text{and}\quad\textbf{1}=(1,1)$$
 In view of (\ref{JF}) and (\ref{IF}) one can write
 \begin{eqnarray*}
 J_{B_2}^{\kappa}(x,y)&=& \frac{ \Gamma(\kappa_2+1/2)}{\Gamma(\kappa_2)\sqrt{\pi}}
 \int_{-1}^1 \; _0F_1 (\mu, A(u), \textbf{1})  (1-u^2)^{\kappa_2-1}\; du
\\&=&  \frac{ \Gamma(\kappa_2+1/2)}{\Gamma(\kappa_2)\sqrt{\pi}}\int_{-1}^1 \; J_{B_2}^{\kappa} ((\sqrt{ 2A_1(u)},\sqrt{ 2A_2(u)}), (\sqrt{2},\sqrt{2})))  (1-u^2)^{\kappa_2-1}\; du
 \end{eqnarray*}
 Using Lemmas \ref{lem3} and  \ref{lem4} \\ \\
 $\displaystyle{J_{B_2}^{\kappa} ((\sqrt{ 2A_1(u)},\sqrt{ 2A_2(u)}), (\sqrt{2},\sqrt{2})))}$
 \begin{eqnarray*}
 &=&   J_{B_2}^{ \kappa'} ( D_1(u), D_1(u), 2,0)
 \\  &=& _0F_1 ( \mu, (D_1(u)^2, D_2(u)^2), (1,0) )
 \\&=&  \frac{\Gamma(\kappa_1+1/2)}{\Gamma(\kappa_1)\sqrt{\pi}}\int_{-1}^1
  \mathcal{I}_\mu\left(\sqrt{2(D_1(u)^2_1+D_2(u)^2+v(D_1(u)^2-D_2(u)^2))}\right) (1-v^2)^{\kappa_1-1} \;dv
    \end{eqnarray*}
 where
 $$ D_1(u)=\sqrt{ A_1(u)} +\sqrt{ A_2(u)},\quad  D_2(u)=-\sqrt{ A_1(u)} +\sqrt{ A_2(u)}$$
 and one can easily check  that
 \begin{eqnarray*}
    2(D_1(u)^2_1+D_2(u)^2+v(D_1(u)^2-D_2(u)^2))
 =  \frac{Z_{x,y}(u,v)}{2}
 \end{eqnarray*}
 This completes the proof of Theorem \ref{th3}.
\end{proof}
The second application  contains  an another proof of the famous known    product formula for the modified bessel function $ \mathcal{I}_\mu$
 (see for example chapter XI of \cite{Wa}).
 In the case $n=2$,    if we let  $x_1=-x_2= x$,  $\lambda=(2,1)$  then  we get  from (\ref{n=2})
 $$\lim_{\ell\rightarrow \infty}\frac{P_{\ell\lambda}( 1+x/ \ell,1-x/ \ell))}{P_{\ell\lambda}(1)}
 =\frac{ \Gamma(2k)}{ 2^{2k-1}\Gamma( k) }
  \int_{-1}^{1} e^{ux}(1-u^2)^{k-1}  \;du.$$
  $$ =\frac{ \Gamma(k+1/2)}{ \sqrt{\pi}\Gamma( k) }
  \int_{-1}^{1} e^{ux}(1-u^2)^{k-1}  \;du= \mathcal{I}_{k-1/2}(x).$$
  Now  using  the product formula (\ref{IF}) with   $x_1=-x_2= x$, $y_1=-y_2= y$ and $\lambda=(2,1)$ and  the following

  $$X_1(\ell,u)X_2(\ell,u)=\left(1+\frac{x}{\ell}\right)\left(1-\frac{x}{\ell}\right)
  \left(1+\frac{y}{\ell}\right)\left(1-\frac{y}{\ell}\right)= 1-\frac{x^2+y^2}{\ell^2}+ \frac{x^2y^2}{\ell^4}$$
     and
   $$\frac{X_1(\ell,u)+X_2(\ell,u)}{2}=1+\frac{uxy}{\ell^2}$$
   $$\frac{X_1(\ell,u)-X_2(\ell,u)}{2}=\sqrt{\left(1+\frac{uxy}{\ell^2}\right)^2-\left( 1-\frac{x^2+y^2}{\ell^2}+\frac{x^2y^2}{\ell^4}\right)}
    = \frac{\sqrt{ x^2+y^2+2uxy}}{\ell}+O(1/\ell^2)$$
    From which it follows that
   $$\lim_{\ell\rightarrow\infty}\left(\frac{X_1(\ell,u)+X_2(\ell,u)}{2}+v\frac{X_1(\ell,u)-X_2(\ell,u)}{2}\right)^{\ell}= e^{v\sqrt{ x^2+y^2+2uxy}} $$
   and\\ \\
   $\displaystyle{\lim_{\ell\rightarrow\infty}\frac{P_\lambda(X_1(\ell,u),X_2(\ell,u))}{P_\lambda^k(1,1)}}$
   \begin{eqnarray*}
   &=& \frac{ \Gamma(k+1/2)}{ \sqrt{\pi}\Gamma( k) } \lim_{\ell\rightarrow\infty}
        (X_1(\ell,u)X_2(\ell,u))^{\ell}
  \int_{-1}^{1} \left(\frac{X_1(\ell,u)+X_2(\ell,u)}{2}+u\frac{X_1(\ell,u)-X_2(\ell,u)}{2}\right)^{\ell}
  \\&&\qquad\qquad\qquad\qquad\qquad\qquad\qquad\qquad\qquad\qquad\qquad\qquad\qquad\qquad(1-u^2)^{k-1}  \;du
  \\& = &\mathcal{I}_{k-1/2}( \sqrt{ x^2+y^2+2uxy}.
   \end{eqnarray*}
  Hence  the following formula  follows
  $$\mathcal{I}_{k-1/2}(x)\mathcal{I}_{k-1/2}(y)=\frac{ \Gamma(k+1/2)}{ \sqrt{\pi}\Gamma( k) }\int_{-1}^{1}\mathcal{I}_{k-1/2}( \sqrt{ x^2+y^2+2uxy}\; (1-u^2)^{k-1}\; du.$$

\par Finally, we hope that  this contribution will shed light on the general case of a product formula for Jack polynomials as in the theorem 
\ref{th1}.

\end{document}